\documentclass[reqno,12pt]{amsart}

\usepackage{hyperref}
\usepackage[mathscr]{eucal}
\usepackage{enumerate}
\numberwithin{equation}{section}

\allowdisplaybreaks[1]

\usepackage{graphicx}
\usepackage{amssymb}
\usepackage[
   includehead,
   includefoot,
     left = 1.5in,
      top = 0.5in,
    right = 1in,
   bottom = 1in
]{geometry}
\usepackage{fancyhdr}
\usepackage{setspace}
\usepackage{calc}
\usepackage{amsthm}
\usepackage{pstricks}
\usepackage{pst-node}
\usepackage{amsmath}

\newcommand{\AutoAbs}[1]{\left\vert #1 \right\vert}
\newcommand{\AutoNorm}[1]{\left\Vert #1 \right\Vert}
\newtheorem{thm}{Theorem}[section]
\newtheorem{lem}{Lemma}[section]

\newcommand{\be}{\begin{eqnarray}}
\newcommand{\ee}{\end{eqnarray}}

\title[Convergence Proof for Godunov Method in GR]
{A Proof Of Convergence for Numerical Approximations Generated by the Locally Inertial Godunov Method in General Relativity}

\author[Z.\ Vogler]{Zeke Vogler}
\thanks{This work summarizes results credited to Vogler's doctoral dissertation, \cite{vogl}, which was supervised by Blake Temple. Both authors were partially supported by second author's NSF Grant, where the problem was first proposed.}

\address{Department of Mathematics\\ University of California\\ Davis, CA 95616\\ USA}
\email{zekius@math.ucdavis.edu}

\author[B.\ Temple]{Blake Temple \\ \\ December 2011}

\address{Department of Mathematics\\ University of California\\ Davis, CA 95616\\ USA}
 \email{temple@math.ucdavis.edu}
 
\begin{document}

\maketitle


\begin{abstract}
We provide details for the proof of convergence of the  {\it locally inertial Godunov method with dynamical time dilation}, introduced by Vogler in \cite{vogl}. \end{abstract}

\section{Introduction} \label{intro}
\setcounter{equation}{0}

 In this paper we fill in the details in the proof of convergence stated in Section 7 of \cite{voglte}, for the {\it locally inertial Godunov method with dynamic time dilation}, a numerical method for computing shock wave solutions of the Einstein equations in Standard Schwarzschild Coordinates (SSC).   We refer the reader to  \cite{voglte} for an introduction and for the notation assumed at the start here.  
 
 The main conclusion of the theorem stated and proved below is that a sequence of approximate solutions $(u_{\Delta x},\mathbf{A}_{\Delta x})\rightarrow (u,\mathbf{A})$ of the locally inertial Godunov scheme that converge boundedly without oscillation to a limit function $(u,\mathbf{A})$, must be a weak solution of the Einstein Euler equations in SSC, c.f. \cite{voglte,groate},
\begin{equation}\label{ch5_conservation_law_source_cmpt}
\begin{split}
    u_{t}+f&(\mathbf{A},u)_{x}=g(\mathbf{A},u,x),\\
       &\mathbf{A}'=h(\mathbf{h},u,x).
\end{split}
\end{equation}
That is, the theorem reduced the proof that one has an exact solution to the two things most easily established in a numerical simulation: convergence without oscillations.
    
System (\ref{ch5_conservation_law_source_cmpt}) is the locally inertial formulation introduced in \cite{groate}, which is weakly equivalent to the Einstein equations, (see \cite{voglte} for discussion).  The proof here is a modification of the Groah and Temple argument used for the locally inertial Glimm scheme \cite{groate}, with several differences.  Both the Glimm scheme and the Godunov methods employ Riemann problem approximations, but the main difference is that the Godunov method employs averaging rather than random sampling at the end of each time step.   The theorem proved here assumes a total variation bound, (it is still an open problem to prove such a bound), while the Groah and Temple theorem establishes this bound by an argument using wave strengths to bound the total variation of waves in the Riemann problem solutions.  In the theorem here, we also allow variable time steps.  The final difference is that here we include right and left boundary data reflecting the limited extent in space of a computer simulation.

There are two main steps in the proof.  The first is to show the discontinuities in the metric $\mathbf{A}$ along the boundary of Riemann cells are accounted for by the inclusion of the term
\begin{equation}
\mathbf{A}'\cdot\nabla_\mathbf{A}f(\mathbf{A}_{ij},\hat{u},x)
\end{equation}
in the ODE step \cite{voglte}.  The second step is to prove the jump in the approximate solution $u_{\Delta x}$ along the time steps are of order $\Delta x$.  In their work \cite{groate}, Groah and Temple did not need the convergence and total variation assumptions because with the Glimm scheme, these assumptions are proven as long as there exists a total variation bound on the initial data, a truly remarkable feature of the scheme.  In this paper, applicable to the simulations in \cite{vogl,voglte}, these assumptions are natural assumptions that can be verified numerically.  In particular, the theorem is perfectly suited to the numerical simulation of points of shock wave interaction in \cite{vogl,voglte}.   Once one numerically establish convergence and a total variation bound, the theorem here implies convergence to a weak solution of the Einstein equations.

\section{The Convergence Theorem}
The main theorem of this paper is the following:
\begin{thm}\label{thm:weak_soln}
Let $u_{\Delta x}(t,x)$ and $\mathbf{A}_{\Delta x}(t,x)$ be the approximate solution generated by the locally inertial Godunov method starting from the initial data $u_{\Delta x}(t_0,x)$ and $\mathbf{A}_{\Delta x}(t_0,x)$ for $t_0>0$.  Assume these approximate solutions exist up to some time $t_{end}>t_0$ and converge to a solution $(u_{\Delta x},\mathbf{A}_{\Delta x})\rightarrow (u,\mathbf{A})$ as $\Delta x\rightarrow 0$ along with a total variation bound at each time step $t_j$
\begin{equation}\label{total_variation_bound}
T.V._{[r_{min},r_{max}]}\{u_{\Delta x}(t_j,\cdot)\}<V,
\end{equation}
where $T.V._{[r_{min},r_{max}]}\{u_{\Delta x}(t_j,\cdot)\}$ represents the total variation of the function $u_{\Delta x}(t_j,x)$ on the interval $[r_{min},r_{max}]$.  Assume the total variation is independent of the time step $t_j$ and the mesh length $\Delta x$.  Then the solution $(u,\mathbf{A})$ is a weak solution to the Einstein equations (1.26)-(1.29) in \cite{groate}.
\end{thm}
\begin{proof}
Suppose we have approximate solutions $(u_{\Delta x},\mathbf{A}_{\Delta x})$ obtained by the locally inertial Godunov method that satisfy the hypothesis of the theorem.  Having a total variation bound at each time $t_j$ places a total variation bound on the inputs to all the Riemann problems posed at that time.  In \cite{groate}, Groah and Temple show a total variation bound on the inputs implies a total variation bound on the solution to the Riemann problem for any time $t$ such that $t_j\leq t<t_{j+1}$.  By the self similarity of the solution to the Riemann problem, this result also implies a total variation bound for any space coordinate within the Riemann cell.  More specifically, we have the following bounds:
\begin{equation}
T.V._{[x_{i-1},x_i]}\{u_{\Delta x}(t,\cdot)\}<V,
\end{equation}
and
\begin{equation}
T.V._{[t_j,t_{j+1})}\{u_{\Delta x}(\cdot,x)\}<V,
\end{equation}
for any $x$ and $t$ within the Riemann cell $R_{ij}$.

All the functions $f$, $G$, and $g$ derived in \cite{groate} are smooth, and it is the metric that is only Lipschitz continuous.  The smoothness of these functions is used throughout this proof.

Let $T=t_{end}-t_0$ be the overall time of the solution, and for each mesh length $\Delta x$ define the minimum time length
\begin{equation}
\Delta t\equiv \min_j\{\Delta t_j\}
\end{equation}
as the minimum over all the time lengths defined in \cite{voglte}.  By definition, this time length is proportional to the mesh length, $\Delta t \propto \Delta x$, implying $O(\Delta t)=O(\Delta x)$, and there exists a constant $C$ bounding all the time lengths, $\Delta t_j < C\Delta t$ for all $j$.  Throughout this section, let $C$ be a generic constant only depending on the bounds for the solution $[t_0,t_{end}]\times [r_{min},r_{max}]$.  This variable is created to unify all the time steps, and more importantly, used to calculate the maximum number of time steps needed to go from $t_0$ to $t_{end}$.

We now follow the development of Groah and Temple in \cite{groate}.  Recall from \cite{voglte}, $u^{RP}_{\Delta x}(t,x)$ denotes the collection of the exact solutions in all the Riemann cells $R_{ij}$ for the Riemann problem of the homogenous system
\begin{equation}
u_{t}+f(\mathbf{A}_{ij},u)_{x}=0.
\end{equation}
So $u^{RP}_{\Delta x}(t,x)$ satisfies the weak form of this conservation law in each Riemann cell
\begin{equation}\label{rp_soln_weak_form}
\begin{split}
0=&\int\int_{R_{ij}}\left\{-u^{RP}_{\Delta x}\varphi_t - f(\mathbf{A}_{ij},u^{RP}_{\Delta x})\varphi_x\right\} dxdt \\
&+ \int_{R_i}\left\{u^{RP}_{\Delta x}(t_{j+1},x)\varphi(t_{j+1},x)-u^{RP}_{\Delta x}(t^+_j,x)\varphi(t_j,x)\right\}dx \\
&+\int_{R_j}\left\{f(\mathbf{A}_{ij},u^{RP}_{\Delta x}(t,x_{i})) \varphi(t,x_{i})\right.\\
&\left.-f(\mathbf{A}_{ij},u^{RP}_{\Delta x}(t,x_{i-1}))\varphi(t,x_{i-1})\right\}dt,
\end{split}
\end{equation}
where $\varphi$ is a smooth test function with $Supp(\varphi)\subset [t_0,t_{end})\times[a,b]$ for $a<r_{min}<r_{max}<b$.

Remember from \cite{voglte}, $\hat{u}(t,u_0)$ denotes the solution to the ODE
\begin{equation}\label{ode_soln}
\begin{split}
\hat{u}_t=G(\mathbf{A}_{ij},\hat{u},x)=g&(\mathbf{A}_{ij},\hat{u},x)- \mathbf{A}'\cdot\nabla_\mathbf{A}f(\mathbf{A}_{ij},\hat{u},x), \\
&\hat{u}(0)=u_0.
\end{split}
\end{equation}
Therefore,
\begin{equation}
\hat{u}(t,u_0)=u_0+\int^t_0\left\{g(\mathbf{A}_{ij},\hat{u}(\xi,u_0),x) -\mathbf{A}'\cdot\nabla_\mathbf{A}f(\mathbf{A}_{ij},\hat{u}(\xi,u_0),x)\right\}d\xi.
\end{equation}
Also, recall from \cite{voglte} $u_{\Delta x}$ denotes the approximate solution obtained using the fractional step method.  Since our fractional method takes the Riemann problem solution and feeds it into the ODE step, $u_{\Delta x}$ is defined on every Riemann cell $R_{ij}$ as
\begin{equation}\label{frac_step_approx_soln}
\begin{split}
u_{\Delta x}(t,x)=u^{RP}_{\Delta x}(t,x) +\int^t_{t_j}&\left\{g(\mathbf{A}_{ij},\hat{u}(\xi-t_j,u^{RP}_{\Delta x}(t,x)),x)\right.\\
-&\frac{\partial f}{\partial \mathbf{A}}\left.(\mathbf{A}_{ij},\hat{u}(\xi-t_j,u^{RP}_{\Delta x}(t,x)))\cdot\mathbf{A}'_{\Delta x}\right\}d\xi.
\end{split}
\end{equation}
This expression implies the error between the approximate solution and the Riemann problem solution is on the order of $\Delta x$; a fact that is repeatedly used throughout the proof.

Define the residual $\varepsilon=\varepsilon(u_{\Delta x},\mathbf{A}_{\Delta x},\varphi)$ of $u_{\Delta x}$ and $\mathbf{A}_{\Delta x}$ as the error of the solution in satisfying the weak form of the conservation law (\ref{ch5_conservation_law_source_cmpt}) by
\begin{equation}\label{residual}
\begin{split}
\varepsilon(u_{\Delta x},\mathbf{A}_{\Delta x},\varphi)\equiv&\int^{r_{max}}_{r_{min}}\int^{t_{end}}_{t_0}\left\{-u_{\Delta x}\varphi_t - f(\mathbf{A}_{\Delta x},u_{\Delta x})\varphi_x - g(\mathbf{A}_{\Delta x},u_{\Delta x},x)\varphi\right\}dxdt\\
&- I_1 - I_2\\
&\sum^{i=n+1}_{i=1,j}\int_{R_{ij}}\left\{-u_{\Delta x}\varphi_t - f(\mathbf{A}_{ij},u_{\Delta x})\varphi_x - g(\mathbf{A}_{ij},u_{\Delta x},x)\varphi\right\}dxdt\\
&- I_1 - I_2,
\end{split}
\end{equation}
where
\begin{equation}
I_1\equiv\int^{r_{max}}_{r_{min}}u_{\Delta x}(t^+_0,x)dx=\sum^{n+1}_{i=1}\int_{R_i}u_{\Delta x}(t^+_0,x)dx,
\end{equation}
and
\begin{equation}
\begin{split}
I_2\equiv&\int^{t_{end}}_{t_0}\left\{f(\mathbf{A}_{ij},u_{\Delta x}(t,r^+_{min}))\varphi(t,r^+_{min}) - f(\mathbf{A}_{ij},u_{\Delta x}(t,r^+_{max}))\varphi(t,r^+_{max})\right\}dt\\
=&\sum_j\int_{R_j}\left\{f(\mathbf{A}_{ij},u_{\Delta x}(t,r^+_{min}))\varphi(t,r^+_{min}) - f(\mathbf{A}_{ij},u_{\Delta x}(t,r^+_{max}))\varphi(t,r^+_{max})\right\}dt,
\end{split}
\end{equation}
The expression $\sum^{i=n+1}_{i=1,j}$ denotes a double sum where the index $i$ runs across all the spatial gridpoints, and the index $j$ runs across all the temporal gridpoints.  Recall from \cite{voglte}, $n$ is the number of spatial gridpoints, and there are $n+1$ Riemann cells.  Our goal is to show $\varepsilon(u_{\Delta x},\mathbf{A}_{\Delta x},\varphi)=O(\Delta x)$ because if the approximation converges $(u_{\Delta x},\mathbf{A}_{\Delta x})\rightarrow (u,\mathbf{A})$ as $\Delta x\rightarrow 0$, then the limit function satisfies the condition of being a weak solution to the Einstein equations $\varepsilon(u,\mathbf{A},\varphi)=0$.

Substituting (\ref{frac_step_approx_soln}) into (\ref{residual}) gives us
\begin{equation}\label{residual_with_soln}
\begin{split}
\varepsilon=&\sum^{i=n+1}_{i=1,j}\int\int_{R_{ij}}\left\{-u^{RP}_{\Delta x}\varphi_t - f(\mathbf{A}_{ij},u_{\Delta x})\varphi_x - g(\mathbf{A}_{ij},u_{\Delta x},x)\varphi\right.\\
&- \varphi_t\int^t_{t_j}\left[g(\mathbf{A}_{ij},\hat{u}(\xi-t_j,u^{RP}_{\Delta x}(t,x)),x)\right.\\
&-\frac{\partial f}{\partial\mathbf{A}}\left.\left.(\mathbf{A}_{ij},\hat{u}(\xi-t_j,u^{RP}_{\Delta x}(t,x)))\cdot\mathbf{A}_{\Delta x}'\right]d\xi\right\}dxdt - I_1 - I_2.
\end{split}
\end{equation}
Define
\begin{equation}
\begin{split}
I^1_{ij}(t,x)\equiv&\int^t_{t_j}\left[g(\mathbf{A}_{ij},\hat{u}(\xi-t_j,u^{RP}_{\Delta x}(t,x)),x)\right.\\
&-\frac{\partial f}{\partial\mathbf{A}}\left.(\mathbf{A}_{ij},\hat{u}(\xi-t_j,u^{RP}_{\Delta x}(t,x)))\cdot\mathbf{A}_{\Delta x}'\right]d\xi
\end{split}
\end{equation}
Plugging the weak form of the conservation law (\ref{rp_soln_weak_form}) of each grid rectangle into (\ref{residual_with_soln}) gives us
\begin{equation}\label{residual_mid_step}
\begin{split}
\varepsilon=&\sum^{i=n+1}_{i=1,j}\int\int_{R_{ij}}\left\{\varphi_x[f(\mathbf{A}_{ij},u^{RP}_{\Delta x})-f(\mathbf{A}_{ij},u_{\Delta x})] - g(\mathbf{A}_{ij},u_{\Delta x},x)\varphi \right. \\
&\phantom{4444444444444444444444}\left.-\varphi_tI^1_{ij}(t,x)\right\}dxdt\\
&-I_1 - \sum^{i=n+1}_{i=1,j}\int_{R_i}\left\{u^{RP}_{\Delta x}(t^-_{j+1},x)\varphi(t_{j+1},x) - u^{RP}_{\Delta x}(t^+_j,x)\varphi(t_j,x)\right\}dx\\
&-I_2 -\sum^{i=n+1}_{i=1,j}\int_{R_j}\left\{f(\mathbf{A}_{ij},u^{RP}_{\Delta x}(t,x_{i}))\varphi(t,x_{i}) - f(\mathbf{A}_{ij},u^{RP}_{\Delta x}(t,x_{i-1}))\varphi(t,x_{i-1})\right\}dt.
\end{split}
\end{equation}
Note
\begin{equation}
\AutoAbs{f(\mathbf{A}_{ij},u^{RP}_{\Delta x})-f(\mathbf{A}_{ij},u_{\Delta x})}\leq C\Delta t
\end{equation}
which implies
\begin{equation}
\begin{split}
&\AutoAbs{\sum^{i=n+1}_{i=1,j}\int\int_{R_{ij}}\varphi[f(\mathbf{A}_{ij},u^{RP}_{\Delta x}) - f(\mathbf{A}_{ij},u_{\Delta x})]dxdt}\\
&\phantom{4444444444444444444}\leq C\AutoNorm{\varphi}_\infty \Delta t^2\Delta x\left(\frac{T}{\Delta t}\right)\left(n+1\right)= O(\Delta x)
\end{split}
\end{equation}
where the number of time steps is proportional to $T/\Delta t$ and the number of space steps is $O(1/\Delta x)$ \cite{voglte}.

Since $u^{RP}_{\Delta x}(t^+_j,x)=u_{\Delta x}(t^+_j,x)$, the following sum is rearranged to become
\begin{equation}
\begin{split}
-I_1 - &\sum^{i=n+1}_{i=1,j}\int_{R_i}\left\{u^{RP}_{\Delta x}(t^-_{j+1},x)\varphi(t_{j+1},x) - u^{RP}_{\Delta x}(t^+_j,x)\varphi(t_j,x)\right\}dx\\
&=\sum_{j\neq0}\int^{r_{max}}_{r_{min}}\left\{u_{\Delta x}(t^+_j,x) - u^{RP}_{\Delta x}(t^-_j,x)\right\}\varphi(t_j,x)dx\\
&=\sum_{j\neq 0}\int^{r_{max}}_{r_{min}}\varphi(t_j,x)\left\{u_{\Delta x}(t^+_j,x) - u_{\Delta x}(t^-_j,x)\right\}dx\\
&+ \sum_{j\neq 0}\int^{r_{max}}_{r_{min}}\varphi(t_j,x)\left\{u_{\Delta x}(t^-_j,x) - u^{RP}_{\Delta x}(t^-_j,x)\right\}dx,
\end{split}
\end{equation}
where the term $u_{\Delta x}(t_j,x)$ is added and subtracted to isolate the jump in the solution $u_{\Delta x}$ across the time step $t_j$.  We define this jump $\varepsilon_1=\varepsilon_1(u_{\Delta x},\mathbf{A}_{\Delta x},\varphi)$ as
\begin{equation}
\varepsilon_1(u_{\Delta x},\mathbf{A}_{\Delta x},\varphi)\equiv\sum_{j\neq0}\int^{r_{max}}_{r_{min}}\varphi(t_j,x)\left\{u_{\Delta x}(t^+_j,x) - u_{\Delta x}(t^-_j,x)\right\}dx,
\end{equation}
and this definition allows us to rewrite (\ref{residual_mid_step}) as
\begin{equation}\label{intro_epsilon1}
\begin{split}
\varepsilon= &\phantom{4}O(\Delta x) + \varepsilon_1 + \sum^{i=n+1}_{i=1,j}\int\int_{R_{ij}}\left\{-g(\mathbf{A}_{ij},u_{\Delta x},x)\varphi - \varphi_tI^1_{ij}(t,x)\right\}dxdt\\
&+ \sum_{j\neq0}\int^{r_{max}}_{r_{min}}\varphi(t,x)\left\{u_{\Delta x}(t^-_j,x) - u^{RP}_{\Delta x}(t^-_j,x)\right\}dx\\
&-I_2 - \sum^{i=n+1}_{i=1,j}\int_{R_j}\left\{f(\mathbf{A}_{ij},u^{RP}_{\Delta x}(t,x_{i}))\varphi(t,x_{i}) - f(\mathbf{A}_{ij}, u^{RP}_{\Delta x}(t,x_{i-1}))\varphi(t,x_{i-1})\right\}dt
\end{split}
\end{equation}
But the last sum is rearranged to cancel the boundary conditions as follows:
\begin{equation}\label{eliminate_i2}
\begin{split}
-I_2 - &\sum^{i=n+1}_{i=1,j}\int_{R_j}\left\{f(\mathbf{A}_{ij},u^{RP}_{\Delta x}(t,x_{i}))\varphi(t,x_{i}) - f(\mathbf{A}_{ij}, u^{RP}_{\Delta x}(t,x_{i-1}))\varphi(t,x_{i-1})\right\}dt\\
=&\sum^{i=n}_{i=1,j}\int_{R_j}\left\{f(\mathbf{A}_{i+1,j},u^{RP}_{\Delta x}(t,x_{i})) - f(\mathbf{A}_{ij},u^{RP}_{\Delta x}(t,x_{i}))\right\}\varphi(t,x_{i})dt\\
+&\sum_{j}\int_{R_j}\left\{f(\mathbf{A}_{1,j},u^{RP}_{\Delta x}(t,x_0)) - f(\mathbf{A}_{1,j},u_{\Delta x}(t,x_0))\right\}\varphi(t,x_0)dt\\
+&\sum_{j}\int_{R_j}\left\{f(\mathbf{A}_{n+1,j},u^{RP}_{\Delta x}(t,x_{n+1})) - f(\mathbf{A}_{n+1,j},u_{\Delta x}(t,x_{n+1}))\right\}\varphi(t,x_{n+1})dt,
\end{split}
\end{equation}
where
\begin{equation}
\begin{split}
&\AutoAbs{\sum_j\int_{R_j}\left\{f(\mathbf{A}_{1,j},u^{RP}_{\Delta x}(t,x_0)) - f(\mathbf{A}_{1,j},u_{\Delta x}(t,x_0))\right\}\varphi(t,x_0)dt}\\
&\leq\AutoNorm{\varphi}_\infty C\Delta t^2\left(\frac{T}{\Delta t}\right)=O(\Delta x),
\end{split}
\end{equation}
and similarly
\begin{equation}
\AutoAbs{\sum_j\int_{R_j}\left\{f(\mathbf{A}_{n+1,j},u^{RP}_{\Delta x}(t,x_{n+1})) - f(\mathbf{A}_{n+1,j},u_{\Delta x}(t,x_{n+1}))\right\}\varphi(t,x_{n+1})dt}=O(\Delta x).
\end{equation}
Note that the resulting double sum in (\ref{eliminate_i2}) lost a term, resulting in only $n$ terms.

To simplify the $I^1_{ij}$ term, we add and subtract a term deviating from it by an order of $\Delta x$, use integration by parts on the new term, and with the result add and subtract another term to reduce the expression further.  To this end, let
\begin{equation}
\begin{split}
I_{\Delta S}\equiv\sum^{i=n+1}_{i=1,j}\int\int_{R_{ij}}\varphi_t\int^t_{t_j} &\left\{g(\mathbf{A}_{ij},\hat{u}(\xi-t_j,u^{RP}_{\Delta x}(\xi,x)),x) - g(\mathbf{A}_{ij},\hat{u}(\xi-t_j,u^{RP}_{\Delta x}(t,x)),x)\right.\\
-\frac{\partial f}{\partial \mathbf{A}}&(\mathbf{A}_{ij},\hat{u}(\xi-t,u^{RP}_{\Delta x}(\xi,x)))\cdot\mathbf{A}_{\Delta x}'\\
+&\frac{\partial f}{\partial \mathbf{A}}\left.(\mathbf{A}_{ij},\hat{u}(\xi-t,u^{RP}_{\Delta x}(t,x)))\cdot\mathbf{A}_{\Delta x}'\right\}d\xi dx dt.
\end{split}
\end{equation}
From the total variation bound on the Riemann problems and the smoothness of $f$, this term is bounded by
\begin{equation}
\begin{split}
\AutoAbs{I_{\Delta S}}&\leq \sum^{i=n+1}_{i=1,j}\int\int_{R_{ij}}\AutoNorm{\varphi_t}_\infty\int^t_{t_j}C\phantom{4}T.V._{[x_{i-1},x_i]}\left\{u_{\Delta x}(\cdot,t_j)\right\}d\xi dxdt\\
&\leq \AutoNorm{\varphi_t}_\infty C\Delta t^2\Delta x\sum_jT.V._{[r_{min},r_{max}]}\{u_{\Delta x}(\cdot,t_j)\}\\
&\leq CV\AutoNorm{\varphi_t}_\infty\Delta x\Delta t^2\frac{T}{\Delta t}=O(\Delta x^2),
\end{split}
\end{equation}
and the above procedure reduces the term to
\begin{equation}\label{splitting_i1_ij}
\begin{split}
&-\int\int_{R_{ij}}\varphi_t I^1_{ij}(t,x)dxdt = I_{\Delta S} - \sum^{i=n+1}_{i=1,j}\int\int_{R_{ij}}\varphi_t\int^t_{t_j}\left\{ g(\mathbf{A}_{ij},\hat{u}(\xi-t_j,u^{RP}_{\Delta x}(\xi,x)),x)\right.\\
&\phantom{444444444}-\frac{\partial f}{\partial\mathbf{A}} \left.(\mathbf{A}_{ij},\hat{u}(\xi-t_j,u^{RP}_{\Delta x}(\xi,x)))\cdot\mathbf{A}_{\Delta x}'\right\}dxdt\\
&=O(\Delta x^2) - \sum^{i=n+1}_{i=1,j}\int_{R_i}\left\{\varphi(t_{j+1},x)\int^{t_{j+1}}_{t_j} \left[g(\mathbf{A}_{ij},\hat{u}(\xi-t_j,u^{RP}_{\Delta x}(\xi,x)),x) \right.\right.\\
&\phantom{444444444}- \frac{\partial f}{\partial\mathbf{A}}\left.(\mathbf{A}_{ij},\hat{u}(\xi-t_j,u^{RP}_{\Delta x}(\xi,x)))\cdot\mathbf{A}_{\Delta x}'\right]d\xi\\
&\phantom{444444444}-\left.\int^{t_{j+1}}_{t_j}\varphi[g(\mathbf{A}_{ij},u_{\Delta x},x) - \frac{\partial f}{\partial\mathbf{A}}(\mathbf{A}_{ij},u_{\Delta x})\cdot\mathbf{A}_{\Delta x}']d\xi\right\}dx\\
&=O(\Delta x^2)-\sum^{i=n+1}_{i=1,j}\int_{R_i}\left\{\varphi(t_{j+1},x)\int^{t_{j+1}}_{t_j}\left[ g(\mathbf{A}_{ij},\hat{u}(\xi-t_j,u^{RP}_{\Delta x}(t_{j+1},x)),x)\right.\right.\\
&\phantom{444444444}-\left.\frac{\partial f}{\partial\mathbf{A}}\left.(\mathbf{A}_{ij},\hat{u}(\xi-t_j,u^{RP}_{\Delta x}(t_{j+1},x)))\cdot\mathbf{A}_{\Delta x}'\right]d\xi\right\}dt+I_4+I_5,
\end{split}
\end{equation}
where
\begin{equation}
\begin{split}
I_4\equiv\sum^{i=n+1}_{i=1,j}&\int_{R_i}\left\{\varphi(t_{j+1},x)\int^{t_{j+1}}_{t_j}\left[ g(\mathbf{A}_{ij},\hat{u}(\xi-t_j,u^{RP}_{\Delta x}(t_{j+1},x)),x)\right.\right.\\
&- g(\mathbf{A}_{ij},\hat{u}(\xi-t_j,u^{RP}_{\Delta x}(\xi,x)),x) - \frac{\partial f}{\partial\mathbf{A}}(\mathbf{A}_{ij},\hat{u}(\xi-t_j,u^{RP}_{\Delta x}(t_{j+1},x)))\cdot\mathbf{A}_{\Delta x}'\\
&+ \left.\frac{\partial f}{\partial\mathbf{A}}\left.(\mathbf{A}_{ij},\hat{u}(\xi-t_j,u^{RP}_{\Delta x}(\xi,x)))\cdot\mathbf{A}_{\Delta x}'\right]d\xi\right\}dx,
\end{split}
\end{equation}
and
\begin{equation}
I_5\equiv\sum^{i=n+1}_{i=1,j}\int\int_{R_{ij}}\varphi\left[g(\mathbf{A}_{ij},u_{\Delta x},x)-\frac{\partial f}{\partial\mathbf{A}}(\mathbf{A}_{ij},u_{\Delta x})\cdot\mathbf{A}_{\Delta x}'\right]dxdt.
\end{equation}
Again by smoothness and the total variation bound, we have
\begin{equation}
\begin{split}
\AutoAbs{I_4}&\leq\AutoNorm{\varphi}_\infty\sum^{i=n+1}_{i=1,j}C\phantom{4}T.V._{[x_{i-1},x_i]}\left\{u_{\Delta x}(\cdot,t_j)\right\}\Delta x\Delta t\\
&\leq\AutoNorm{\varphi}_\infty C\Delta x\Delta t\sum_jT.V._{[r_{min},r_{max}]}\left\{u_{\Delta x}(\cdot,t_j)\right\}
=\AutoNorm{\varphi}_\infty CV\Delta x\Delta t\frac{T}{\Delta t}=O(\Delta x).
\end{split}
\end{equation}
Substituting (\ref{eliminate_i2}) and (\ref{splitting_i1_ij}) into (\ref{intro_epsilon1}) along with using (\ref{frac_step_approx_soln}) as an identity leaves us with
\begin{equation}\label{f_jumps_cancel}
\begin{split}
\varepsilon=&\phantom{4}O(\Delta x) + \varepsilon_1 - \sum^{i=n+1}_{i=1,j}\int\int_{R_{ij}}\varphi\frac{\partial f}{\partial\mathbf{A}}(\mathbf{A}_{ij},u_{\Delta x})\cdot\mathbf{A}_{\Delta x}'dxdt\\
&+\sum^{i=n}_{i=1,j}\int_{R_j}\varphi(t,x_{i})\left\{f(\mathbf{A}_{i+1,j},u^{RP}_{\Delta x}(t,x_{i})) - f(\mathbf{A}_{ij},u^{RP}_{\Delta x}(t,x_{i}))\right\}dt
\end{split}
\end{equation}
The second sum represents the jump in the flux function $f$, resulting from the discontinuities in the metric $\mathbf{A}$, and the first sum is the addition to the ODE step (\ref{ode_soln}) specifically designed to cancel these jumps in the flux.

To see how the cancelation works, we perform a Taylor expansion on the test function, and we add and subtract terms deviating by order $\Delta x$.  The first sum in (\ref{f_jumps_cancel}) is expanded as
\begin{equation}\label{df_da_expansion}
\begin{split}
\sum^{i=n+1}_{i=1,j}\int\int_{R_{ij}}&\varphi\frac{\partial f}{\partial\mathbf{A}}(\mathbf{A}_{ij},u_{\Delta x})\cdot\mathbf{A}_{\Delta x}'dxdt\\
&=\sum^{i=n+1}_{i=1,j}\int\int_{R_{ij}}\varphi(x_{i},t)\frac{\partial f}{\partial\mathbf{A}}(\mathbf{A}_{ij},u_{\Delta x})\cdot\mathbf{A}_{\Delta x}'dxdt+O(\Delta x)\\
&=\sum^{i=n+1}_{i=1,j}\int_{R_j}\varphi(x_{i},t)\int_{R_i}\left\{\frac{\partial f}{\partial\mathbf{A}}(\mathbf{A}_{ij},u_{\Delta x})\cdot\mathbf{A}_{\Delta x}'-\frac{\partial f}{\partial\mathbf{A}}(\mathbf{A}_{ij},u^{RP}_{\Delta x})\cdot\mathbf{A}_{\Delta x}'\right\}dxdt\\
&+\sum^{i=n+1}_{i=1,j}\int_{R_j}\varphi(x_{i},t)\int_{R_i}\left\{\frac{\partial f}{\partial\mathbf{A}}(\mathbf{A}_{ij},u^{RP}_{\Delta x})\cdot\mathbf{A}_{\Delta x}'-\frac{\partial f}{\partial\mathbf{A}}(\mathbf{A}_{ij},u^{RP}_{\Delta x}(x_i,t))\cdot\mathbf{A}_{\Delta x}'\right\}dxdt\\
&+\sum^{i=n+1}_{i=1,j}\int_{R_j}\varphi(x_{i},t)\int_{R_i}\left\{\frac{\partial f}{\partial\mathbf{A}}(\mathbf{A}_{ij},u^{RP}_{\Delta x}(x_i,t))\cdot\mathbf{A}_{\Delta x}'\right.\\
&\phantom{4444444444}\left.-\frac{\partial f}{\partial\mathbf{A}}(\mathbf{A}_{\Delta x}(x+\frac{\Delta x}{2},t_j),u^{RP}_{\Delta x}(x_i,t))\cdot\mathbf{A}_{\Delta x}'\right\}dxdt\\
&+\sum^{i=n+1}_{i=1,j}\int_{R_j}\varphi(x_{i},t)\int^{x_i}_{x_{i-1}}\frac{\partial f}{\partial\mathbf{A}}(\mathbf{A}_{\Delta x}(x+\frac{\Delta x}{2},t_j),u^{RP}_{\Delta x}(x_i,t))\cdot\mathbf{A}_{\Delta x}'dxdt+O(\Delta x)
\end{split}
\end{equation}
From the smoothness of $f$, each of the first three sums in equation (\ref{df_da_expansion}) are $O(\Delta x)$ for the following reasons: the first sum is order $\Delta x$ from the ODE step in the definition of the approximate solution $u_{\Delta x}$ (\ref{frac_step_approx_soln}), the second sum is order $\Delta x^2$ by the total variation bound on solutions to the Riemann problems, and the third sum is order $\Delta x$ by the Lipschitz continuity of the metric $\mathbf{A}$.  After these bounds are established, (\ref{df_da_expansion}) reduces to
\begin{equation}\label{df_da_result}
\begin{split}
\sum^{i=n+1}_{i=1,j}&\int\int_{R_{ij}}\varphi\frac{\partial f}{\partial\mathbf{A}}(\mathbf{A}_{ij},u_{\Delta x})\cdot\mathbf{A}_{\Delta x}'dxdt\\
&=\sum^{i=n+1}_{i=1,j}\int_{R_j}\varphi(x_{i},t)\int^{x_i}_{x_{i-1}}\frac{\partial f}{\partial\mathbf{A}}(\mathbf{A}_{\Delta x}(x+\frac{\Delta x}{2},t_j),u^{RP}_{\Delta x}(x_i,t))\cdot\mathbf{A}_{\Delta x}'dxdt+O(\Delta x)\\
&=\sum^{i=n+1}_{i=1,j}\int_{R_j}\varphi(x_{i},t)\int^{x_i}_{x_{i-1}}\frac{\partial f}{\partial x}(\mathbf{A}_{\Delta x}(x+\frac{\Delta x}{2},t_j),u^{RP}_{\Delta x}(x_i,t))dxdt+O(\Delta x)\\
&=\sum^{i=n+1}_{i=1,j}\int_{R_j}\varphi(t,x_{i})\left\{f(\mathbf{A}_{i+1,j},u^{RP}_{\Delta x}(t,x_{i+})) - f(\mathbf{A}_{ij},u^{RP}_{\Delta x}(t,x_{i+}))\right\}dt + O(\Delta x).
\end{split}
\end{equation}
Plugging this result (\ref{df_da_result}) into (\ref{f_jumps_cancel}) gives us
\begin{equation}
\begin{split}
\varepsilon=&\phantom{4}O(\Delta x) + \varepsilon_1\\
&-\sum_{j}\int_{R_j}\varphi(t,x_{n+1})\left\{f(\mathbf{A}_{n+2,j},u^{RP}_{\Delta x}(t,x_{n+1})) - f(\mathbf{A}_{n+1,j},u^{RP}_{\Delta x}(t,x_{n+1}))\right\}dt,
\end{split}
\end{equation}
where one term remains due to the mismatch in the number of terms in the spatial sum.  Clearly, this last term is $O(\Delta x)$.

So the residual boils down to
\begin{equation}
\varepsilon(u_{\Delta x},\mathbf{A}_{\Delta x},\varphi)=\varepsilon_1(u_{\Delta x},\mathbf{A}_{\Delta x},\varphi) + O(\Delta x),
\end{equation}
with all that remains to show is
\begin{equation}
\varepsilon_1=\sum_{j\neq0}\int^{r_{max}}_{r_{min}}\varphi(t_j,x)\left\{u_{\Delta x}(t^+_j,x) - u_{\Delta x}(t^-_j,x)\right\}dx=O(\Delta x).
\end{equation}

To estimate $\varepsilon_1$, we break up the sum by each time step $t_j$ and define
\begin{equation}\label{epsilon_j_1_defn}
\begin{split}
\varepsilon^j_1&\equiv\int^{r_{max}}_{r_{min}}\varphi(t_j,x)\left\{u_{\Delta x}(t^+_j,x) - u_{\Delta x}(t^-_j,x)\right\}dx\\
&=\sum_i\int^{x_{i+}}_{x_i-}\varphi(t_j,x)\left\{u_{\Delta x}(t^+_j,x) - u_{\Delta x}(t^-_j,x)\right\}dx,
\end{split}
\end{equation}
with $x_{i+}\equiv x_{i+\frac{1}{2}}$ and $x_{i-}\equiv x_{i-\frac{1}{2}}$.

Recall from \cite{voglte}, the approximate solution for the new time step $t^+_j$ is computed by the Godunov step, using averages at the top of each Riemann cell $R_{ij}$.  In particular, the solution at each new time step is
\begin{equation}
u_{\Delta x}(t^+_j,x)\equiv\hat{u}(t_j-t_{j-1},\bar{u}(t_j),x))
\end{equation}
where
\begin{equation}
\bar{u}(t_j)\equiv\frac{1}{\Delta x}\int^{x_{i+}}_{x_i-}u^{RP}_{\Delta x}(t_j,x)dx
\end{equation}

To finish the proof, a lemma is needed, which is proven at the end of this section.  This lemma states the difference of the ODE step taken on an average verses the solution to the Riemann problem across the top of the Riemann cell is bounded by the total variation of the Riemann problem.
\begin{lem}\label{lem:avg_vs_rp_estimate}
Let $u^{RP}_{\Delta x}$ represent the solution of the Riemann problem in the Riemann cell $R_{i,j-1}$ and $\bar{u}_{\Delta x}(t)$ denote the average of the Riemann problem solution across Riemann cell.  Let $\hat{u}$ be the solution obtained by the ODE step (\ref{ode_soln}) and $\varphi$ be a smooth test function.  Then the following bound holds
\begin{equation}\label{avg_vs_rp_ineq}
\begin{split}
&\AutoAbs{\int^{x_{i+}}_{x_{i-}}\left\{\hat{u}(t_j-t_{j-1},\bar{u}_{\Delta x}(t_j),x)-\hat{u}(t_j-t_{j-1},u^{RP}_{\Delta x}(t_j,x),x)\right\}\varphi(t_j,x)dx}\\
&\phantom{44444444444444}\leq\phantom{4} C\AutoNorm{\varphi}_\infty\Delta x\Delta t \phantom{4}T.V._{[x_i,x_{i+1}]}\{u_{\Delta x}(t_j,\cdot)\}
\end{split}
\end{equation}
for some constant C.
\end{lem}
Using Lemma \ref{lem:avg_vs_rp_estimate}, (\ref{epsilon_j_1_defn}) is rewritten as solutions to the ODE step (\ref{ode_soln}) and bounded by
\begin{equation}
\begin{split}
\varepsilon^j_1&=\sum_i\int^{x_{i+}}_{x_{i-}}\varphi(t_j,x)\left\{ \hat{u}(t_j-t_{j-1},\bar{u}(t_j),x) - \hat{u}(t_j-t_{j-1},u^{RP}_{\Delta x}(t_j,x),x)\right\}dx\\
&\leq C\AutoNorm{\varphi}_\infty\Delta x\Delta t\sum_i T.V._{[x_{i-},x_{i+}]}\{u^{RP}_{\Delta x}(\cdot,t_j)\} = C\AutoNorm{\varphi}_\infty\Delta x\Delta t \phantom{4}T.V._{[r_{min},r_{max}]}\{u^{RP}_{\Delta x}(\cdot,t_j)\}
\end{split}
\end{equation}
By the total variation bound on $u_{\Delta x}(t_j,\cdot)$, the residual is bounded by
\begin{equation}
\varepsilon_1\leq\sum_{j\neq0}C\AutoNorm{\varphi}_\infty\Delta x\Delta t \phantom{4}T.V._{[r_{min},r_{max}]}\{u^{RP}_{\Delta x}(\cdot,t_j)\} \leq C\frac{T}{\Delta t}\Delta x\Delta tV = O(\Delta x).
\end{equation}
Therefore, $\varepsilon=O(\Delta x)$ and the proof is complete.
\end{proof}

To prove Lemma \ref{lem:avg_vs_rp_estimate}, a preliminary result is needed: given a function on a set of points the difference of the function between any point and the average is bounded by the total variation of that function on the set.  This result is provided by the following
\begin{lem}\label{lem:avg_bound}
Let $u(x)$ be a function on the set $[x_{i-},x_{i+}]$ and
\begin{equation}\label{avg_defn}
\bar{u}=\frac{1}{\Delta x}\int^{x_{i+}}_{x_{i-}}u(x)dx
\end{equation}
be the average of $u$ on this set.  Then we have
\begin{equation}\label{avg_bnded_by_tv}
|\bar{u}-u(x)|\leq\sup_{x_1,x_2\in[x_i,x_{i+1}]}|u(x_1)-u(x_2)|\leq T.V._{[x_{i-},x_{i+}]}\{u(\cdot)\}.
\end{equation}

\end{lem}
\begin{proof}
The second inequality is true by the definition of the total variation
\begin{equation}
\sup_{x_1,x_2\in[x_i,x_{i+1}]}|u(x_1)-u(x_2)|\leq T.V._{[x_{i-},x_{i+}]}\{u(\cdot)\}.
\end{equation}
To prove the first inequality, we assume it is false to obtain a contradiction, so suppose there exists $x_*\in[x_{i-},x_{i+}]$ such that
\begin{equation}\label{avg_inequality}
\sup_{x_1,x_2\in[x_i,x_{i+1}]}|u(x_1)-u(x_2)|<|\bar{u}-u(x_*)|.
\end{equation}
Relabel the u-coordinates by an isometry $\varphi:u\rightarrow v$ that maps the point $u(x_*)$ to the origin in the v-coordinates (i.e. $\varphi(u(x_*)) = \mathbf{0}$), and the vector $\bar{u}-u(x_*)$ in the direction of the 1st coordinate $v^1$, as show in Figure \ref{fig:avg_map}.

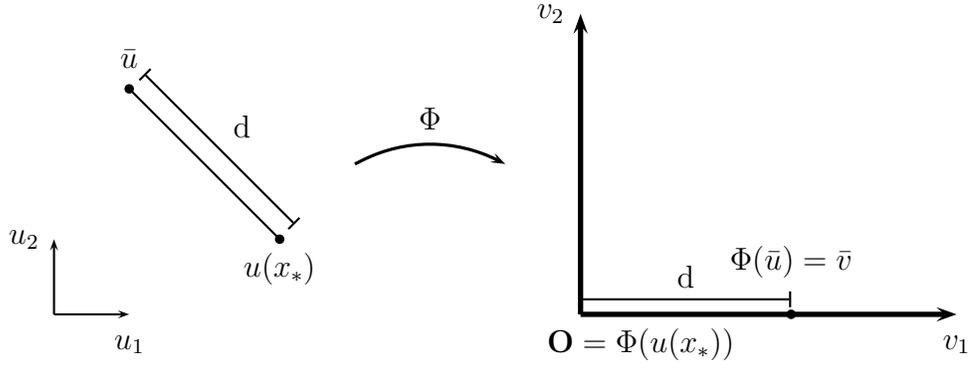
\begin{figure}[!h]
\begin{pspicture}(12,4)(0,-0.5)
\psline{->}(0,0)(1,0)
\psline{->}(0,0)(0,1)
\rput(-0.4,1){$u_2$}
\rput(1,-0.4){$u_1$}
\psline{*-*}(1,3)(3,1)
\psline{|-|}(1.2,3.2)(3.2,1.2)
\rput(1,3.4){$\bar{u}$}
\rput(3,0.6){$u(x_*)$}
\rput(2.5,2.5){d}
\psdots(9.8,0)
\psline{|-|}(7,0.2)(9.8,0.2)
\rput(8.4,0.5){d}
\rput(7.8,-0.4){$\mathbf{O}=\Phi(u(x_*))$}
\rput(9.8,0.7){$\Phi(\bar{u})=\bar{v}$}
\psline[linewidth=1.2pt,linearc=2]{->}(4,2)(5,2.5)(6,2)
\rput(5,2.6){$\Phi$}
\psline[linewidth=2pt]{->}(7,0)(12,0)
\psline[linewidth=2pt]{->}(7,0)(7,4)
\rput(6.6,4){$v_2$}
\rput(12,-0.4){$v_1$}
\end{pspicture}\caption{The isometry $\Phi:u\rightarrow v$}
\label{fig:avg_map}
\end{figure}

Since the average of a collection of points is independent of the coordinate system in which they are labeled in, we have
\begin{equation}
\bar{v}\equiv\frac{1}{\Delta x}\int^{x_{i+}}_{x_{i-}}v(x)dx=\frac{1}{\Delta x}\int^{x_{i+}}_{x_{i-}}\varphi(u(x))dx=\varphi\left(\frac{1}{\Delta x}\int^{x_{i+}}_{x_{i-}}u(x)dx\right)=\varphi(\bar{u})
\end{equation}
The following inequality holds by transforming equation (\ref{avg_inequality}) over to v-coordinates
\begin{equation}
|v(x)|=|u(x)-u(x_*)|<|\bar{u}-u(x_*)|=|\bar{v}|\phantom{4444}\forall x\in[x_{i-},x_{i+}],
\end{equation}
which implies
\begin{equation}
|\bar{v}|=\AutoAbs{\frac{1}{\Delta x}\int^{x_{i+1}}_{x_i}v(x)dx}\leq \frac{1}{\Delta x}\int^{x_{i+1}}_{x_i}|v(x)|dx< \frac{1}{\Delta x}\int^{x_{i+1}}_{x_i}|\bar{v}|dx =|\bar{v}|.
\end{equation}
This inequality $|\bar{v}|<|\bar{v}|$ is an obvious contradiction, proving the first inequality in (\ref{avg_bnded_by_tv}).
\end{proof}
Now we prove the lemma used in the proof of Theorem \ref{thm:weak_soln}

\vspace{.2cm}
\noindent{\bf Proof of Lemma \ref{lem:avg_vs_rp_estimate}}
Recall the solution to the ODE step has the form:
\begin{equation}
\hat{u}(t_j-t_{j-1},u^{RP}_{\Delta x}(t_j,x),x)=u^{RP}_{\Delta x}(t,x)+\int^{t_j}_{t_{j-1}}G(\mathbf{A}_{ij},u^{RP}_{\Delta x}(t,x),x)dt.
\end{equation}
This solution implies the LHS of (\ref{avg_vs_rp_ineq}) is written out as
\begin{equation}
\begin{split}
&\AutoAbs{\int^{x_{i+}}_{x_{i-}}\left\{\hat{u}(t_j-t_{j-1},\bar{u}_{\Delta x}(t_j),x)-\hat{u}(t_j-t_{j-1},u^{RP}_{\Delta x}(t_j,x),x)\right\}\varphi(t_j,x)dx}\\
=&\left\vert\int^{x_{i+}}_{x_{i-}}\left\{(\bar{u}_{\Delta x}(t_j)-u^{RP}_{\Delta x}(t_j,x))\right.\right. \\
&\left.+\int^{t_j}_{t_{j-1}}\left.(G(\mathbf{A}_{ij},\bar{u}_{\Delta x}(t),x)-G(\mathbf{A}_{ij},u^{RP}_{\Delta x}(t,x),x))dt\right\}\varphi(t_j,x)dx\right\vert \\
=&\left\vert\int^{x_{i+}}_{x_{i-}}\left\{\bar{u}_{\Delta x}(t_j)-u^{RP}_{\Delta x}(t_j,x)\right\}\varphi(t_j,x_i)dx\right.\\ &\left.+\int^{x_{i+}}_{x_{i-}}\int^{t_j}_{t_{j-1}}\left\{G(\mathbf{A}_{ij},\bar{u}_{\Delta x}(t_j),x)- G(\mathbf{A}_{ij},u^{RP}_{\Delta x}(t_j,x),x)\right\}dt\phantom{4}\varphi(t_j,x_i)dx\right\vert \\
& + O(\Delta x^2),
\end{split}
\end{equation}
where the test function in the first term is approximated by a Taylor expansion.  By the definition of the average function $\bar{u}$, the first term is zero.  By the smoothness of $G$, the bound (\ref{avg_vs_rp_ineq}) is proven by
\begin{equation}
\begin{split}
&\AutoAbs{\int^{x_{i+}}_{x_{i-}}\left\{\hat{u}(t_j-t_{j-1},\bar{u}_{\Delta x}(t_j),x)-\hat{u}(t_j-t_{j-1},u^{RP}_{\Delta x}(t_j,x),x)\right\}\varphi(t_j,x)dx}\\
&\phantom{4444}\leq\phantom{4}C\AutoNorm{\varphi}_\infty\Delta x\Delta t\sup_{x_{i-}<x<x_{i+}}\{\AutoAbs{\bar{u}_{\Delta x}(t_j)-u^{RP}_{\Delta x}(t_j,x)}\}\\
&\phantom{4444}\leq\phantom{4}C\AutoNorm{\varphi}_\infty\Delta x\Delta t \phantom{4}T.V._{[x_{i-},x_{i+}]}\{u_{\Delta x}(t_j,\cdot)\},
\end{split}
\end{equation}
where Lemma \ref{lem:avg_bound} is used to bound the difference between the average and the solution to the Riemann problem.
\begin{flushright}
$\Box$
\end{flushright}

\end{document}